\documentclass[12pt,reqno]{amsart}
\usepackage[utf8x]{inputenc}

\usepackage{latexsym,amssymb,amsmath,amsthm,bbm}
\usepackage{epsfig}
\usepackage{pst-eucl, esint}
\usepackage{pst-plot,pstricks-add}

\newcommand{\R}{{\ensuremath{\mathbb{R}}}}
\newcommand{\N}{{\ensuremath{\mathbb{N}}}}

\renewcommand{\P}{\ensuremath{\mathbb{P}}}
\renewcommand{\dj}{d\kern-0.4em\char"16\kern-0.1em}
\newcommand{\E}{\ensuremath{\mathbb{E}}}

\newcommand{\ds}[1]{{\displaystyle{#1}}}

\newtheorem{Thm}{Theorem}[section]

\newtheorem{Lem}[Thm]{Lemma}
\newtheorem{Prop}[Thm]{Proposition}
\newtheorem{Rem}[Thm]{Remark}

\numberwithin{equation}{section}
\title[Nondegenerate jump processes]{Analysis of jump processes with nondegenerate jumping kernels}
\author{Moritz Kassmann}
\address{Fakult\"{a}t f\"{u}r Mathematik, Universit\"{a}t Bielefeld, Postfach 100131, D-33501 Bielefeld, Germany}
\curraddr{}
\email{moritz.kassmann@math.uni-bielefeld.de}
\thanks{}

\author{Ante Mimica}
\address{Fakult\"{a}t f\"{u}r Mathematik, Universit\"{a}t Bielefeld, Postfach 100131, D-33501 Bielefeld, Germany}
\curraddr{}
\email{amimica@math.uni-bielefeld.de}
\thanks{}



\date{\today}

\begin{document}

%\author{Moritz Kassmann\\ Department of Mathematics\\Bielefeld University\\Germany\\ 
%\and 
%Ante Mimica\\Department of Mathematics\\Bielefeld University\\Germany\\ }

\begin{abstract}
 We prove regularity estimates for functions which are harmonic with respect to
certain jump processes. The aim of this article is to extend the method of
Bass-Levin\cite{BL1} and Bogdan-Sztonyk\cite{BoSt05} to more general processes. Furthermore, we establish a
new version of the Harnack inequality that implies regularity estimates
for corresponding harmonic functions. 
\end{abstract}

\maketitle

\section{Introduction}
Let $\alpha \in (0,2)$. We define a non-local operator $\mathcal{L}$  by
\begin{align}\label{eq:defL}
	\mathcal{L}f(x)=\int_{\R^d\setminus\{0\}}(f(x+h)-f(x)-\langle
\nabla f(x), h \rangle \,\mathbbm{1}_{\{|h|\leq 1\}})n(x,h)\,dh,
\end{align}
for $f\in C_b^2(\R^d)$. 
Here $n\colon \R^d\times \left(\R^d\setminus \{0\}\right)\rightarrow [0,\infty)$
is a measurable function with 
\begin{align} \label{eq:n_assum_bl}
c_1 |h|^{-d-\alpha} \leq n(x,h) \leq c_2 |h|^{-d-\alpha} 
\end{align}
for every $h \in \R^d\setminus \{0\}$, any $x \in \R^d$ and fixed positive
reals $c_1 < c_2$. Note that $n(x,h) = |h|^{-d-\alpha}$ for every $h$ implies
$\mathcal{L}f = - c(\alpha) (-\Delta)^{\alpha/2} f$ with some appropriate
constant $c(\alpha)$. 

In \cite{BL1} it is shown that harmonic functions with respect to
$\mathcal{L}$ satisfy a Harnack inequality in the following sense: There is a
constant $c_3\geq 1$ such that for every ball $B_R$ the following
implication holds:
\begin{align*}
 f \geq 0 \text{ in } \R^d, \; f \text{ harmonic in } B_R \quad
\Rightarrow \quad \forall\, x,y \in B_{R/2}: f(x) \leq c_3 f(y) \,.  
\end{align*}
In \cite{BL1} it is also shown that harmonic functions with respect to
$\mathcal{L}$ satisfy the following a-priori estimate:  There are constants
$\beta \in (0,1)$, $c_4\geq 1$ such that for every ball $B_R$ the following
implication holds:
\begin{align*}
f \text{ harmonic in } B_R \quad \Rightarrow \quad
\|f\|_{C^\beta(\overline{B_{R/2}})} \leq c_4 \|f\|_\infty \,.
\end{align*}

This result and its proof recently generated several research activities, see the short
discussion below. Our aim is to prove similar results under weaker
assumptions on the kernel $n$. 

Let us be more precise. We consider kernels $n\colon \R^d\times \left(\R^d\setminus
\{0\}\right)\rightarrow [0,\infty)$ that satisfy for every $x, h\in\R^d$, $h\not=0 $
\begin{equation}\label{eq:symm_ass}
	n(x,h)=n(x,-h)
\end{equation}
and
%\begin{align}\label{eq:main_assum}
% k\big(\tfrac{h}{|h|}\big) j(|h|) \leq n(x,h)\leq j(|h|) \,,
%\end{align}
%where 
%%$C_0\geq 1$ is a constant and 
%$k\colon S^{d-1}\rightarrow [0,\infty)$ is a
%bounded symmetric function on the unit sphere such that there exists a constant
%$\delta>0$ and a non-empty open set $I\subset S^{d-1}$ with 
%AM:
\begin{align}\label{eq:main_assum}
k_1\big(\tfrac{h}{|h|}\big) j(|h|) \leq n(x,h)\leq k_2\big(\tfrac{h}{|h|}\big) j(|h|)\,
\end{align}
where 
$k_1,k_2\colon S^{d-1}\rightarrow [0,\infty)$ are measurable bounded symmetric functions on the unit sphere satisfying the following conditions: There are $\delta>0, N \in \mathbb{N}, \varepsilon_1, \ldots, \varepsilon_N > 0$ and $\eta_1, \ldots, \eta_N \in S^{d-1}$ such that for $S_i = S^{d-1} \cap \big(B(\eta_i, \varepsilon_i) \cup B(-\eta_i, \varepsilon_i)\big)$
\begin{align}\label{eq:k}
k_2(\xi) \geq k_1(\xi) \geq \delta \; \text{ if } \xi \in \bigcup\limits_{i=1}^N S_i \quad \text{ and } \quad k_2(\xi) = k_1(\xi) = 0 \text{ otherwise}. 
\end{align}

Let $j:(0,\infty) \to [0,\infty)$ be a function 
such that $\int_{\R^d}
(|z|^2 \wedge 1) \, j(|z|)\,dz $ is finite. We assume further: 
\begin{itemize}
	\item[(J1)] There exists $\alpha \in (0,2)$ and a function $\ell\colon (0,2)\rightarrow (0,\infty)$ which is slowly varying
	%\footnotemark\, 
	at $0$ (i.e. $\ds{\lim_{r\to 0+}\frac{\ell(\lambda r)}{\ell(r)}=1}$ for any $\lambda>0$) and bounded away from $0$ and $\infty$ on every compact interval such that  
		\[ j(t) = \frac{\ell(t)}{t^{d+\alpha}} \; \text{ for every } 0<t\leq 1 \,. \]
	\item[(J2)] There is a constant $\kappa \geq 1$ such that
		\begin{align*}
			&j(t) \leq \kappa j(s) \; \text{ whenever }\ 1\leq s\leq
t \,.
		\end{align*}
\end{itemize}
In order to establish regularity estimates we need an additional weak 
assumption. 
\begin{itemize} 
	\item[(J3)] There is $\sigma > 0$ such that 
    \[ \limsup_{R \to \infty} R^\sigma \int\limits_{|z|>R} j(|z|) \, dz \leq 1
\,.
    \]
If this condition holds, then one can always choose $\sigma \in
(0,\alpha)$.
\end{itemize}

\begin{Rem}
	The symmetry assumption (\ref{eq:symm_ass}) is used only in Proposition \ref{prop:estp} and can be dispensed with if $\alpha\in (0,1)$.
\end{Rem}
\vspace*{-3ex}
{\bf Example 1:} If a kernel $n$ satisfies condition (\ref{eq:n_assum_bl}),
then it also satisfies (J1)-(J3). Choose $N=1$, $\varepsilon_1=4$, i.e. $S_1=S^{d-1}$, $k_1 \equiv \delta=c_1$, $k_2 \equiv c_2$, 
$j(s)=s^{-d-\alpha}$ in
(\ref{eq:main_assum}), $\ell \equiv 1$ in (J1), $\kappa = 1$ in (J2) and
$\sigma \in (0,\alpha)$ arbitrarily in (J3). In general, (J1)-(J3) hold for jumping kernels corresponding to stable
processes, stable-like processes and truncated versions. Sums of such jumping
kernels can be considered, too. 

{\bf Example 2:} Let $N \in \N$, $\eta_1, \ldots, \eta_N \in S^{d-1}$ and $\varepsilon_1, \ldots, \varepsilon_N$ be positive real numbers such that the sets $S_i = S^{d-1} \cap \big(B(\eta_i,  \varepsilon_i) \cup B(-\eta_i, \varepsilon_i)\big)$ are pairwise disjoint for $i=1, \ldots, N$. Set $B=\bigcup\limits_{i=1}^N S_i$. Let $k_1=\delta \mathbbm{1}_B$ for some $\delta >0$ and $k_2 = c k_1$ for some $c>1$. Let $j(s)=s^{-d-\alpha}$ for $s>0$. Then our assumptions are satisfied if (\ref{eq:main_assum}) and (\ref{eq:symm_ass}) hold true. For the particular choice where $x \mapsto n(x,h)$ is constant (case of L\'{e}vy process), this class of examples is treated in \cite[p.148]{BoSt05}, where it is shown that for $N = \infty$ the Harnack inequality fails.

Given a linear operator $\mathcal{L}$ as in
(\ref{eq:defL}) we assume that there exists a strong Markov process
$X=(X_t,\P^x)$ with paths that are right-continuous with left limits such that
the process
\[
\left\{f(X_t)-f(X_0)-\int_0^t\mathcal{L}f(X_s)\,ds\right\}_{t\geq 0}
\]
is a $\P^x$-martingale for all $x\in\R^d$ and $f\in C_b^2(\R^d)$. We say that a
bounded function $f:\R^d \to \R$ is harmonic with respect to $\mathcal{L}$ in an open set $\Omega$ if $f\big(X_{\min(t,\tau_{\Omega'})}\big)$ is a right-continuous martingale for every open $\Omega' \subset \R^d \text{ with } \overline{\Omega'} \subset \Omega$.

We can prove the following version of the Harnack inequality.

\begin{Thm}\label{thm:harnack}
Assume (J1) and (J2). There exist constants $c_1, c_2 \geq 1$ such that for
every
$x_0\in \R^d$, $r\in
(0,\frac{1}{4})$ and every bounded function $f\colon \R^d\rightarrow \R$ which
is
non-negative in $B(x_0,4r)$  and harmonic in $B(x_0,4r)$ the following
estimate holds
\[
f(x)\leq c_1 f(y)+c_2\left( \frac{r^\alpha}{\ell(r)}\right) \sup_{v\in
B(x_0,2r)}\int_{B(x_0,4r)^c}f^-(z)n(v,z-v)\,dz\ 
\]
for all $x,y\in B(x_0,r)$.
\end{Thm}

\begin{Rem}
 If $f$ is, in addition, non-negative in all of $\R^d$, then the
classical version of the Harnack inequality follows, i.e. for all $x,y\in
B(x_0,r)$:
\[
f(x)\leq c_1 f(y) \,.
\]
\end{Rem}
As a corollary to the Harnack inequality we obtain the following regularity
result.

\begin{Thm}\label{thm:reg_est}
Assume (J1), (J2) and (J3). Then there exist $\beta \in (0,1)$, $c_3,c_4 \geq 1$ such that for every $x_0
\in \R^d$, every $R  \in (0,1)$, every function
$f:\R^d \to \R$ which is harmonic in $B(x_0,R)$ and every $\rho \in (0,R/2)$
\begin{align}
\sup\limits_{x,y \in B(x_0,\rho)} | f(x) - f(y) | &\leq c_3 \|f\|_\infty
(\rho/R)^{\beta} \,, \\
\text{ in particular } \quad \|f\|_{C^\beta(\overline{B(x_0,R/2)})} &\leq
c_4 \|f\|_\infty \,. 
\end{align}
\end{Thm}

Let us comment on the differences between our results and those of \cite{BL1}:

(1) We can treat kernels $n(x,h)$ for which the quantity
\[ \inf\limits_{x \in \R^d} \liminf\limits_{r\to 0+}  \frac{|\{h \in B(0,r);
n(x,h) = 0\}|}{|B(0,r)|} \]
is arbitrarily close to $1$, e.g. $n(x,h)$ as in (\ref{eq:n_cone}). \\ 
(2) For fixed $x \in \R^d$, upper and lower bounds for $n(x,h)$ may not allow
for scaling.\\
(3) Large jumps of the process might not be comparable, i.e. the quantity 
\begin{align*} 
\sup \left\{ \frac{n(x,h_1)}{n(y,h_2)} ; |x-y| \leq 1, |h_1 - h_2| \leq 1,
|h_2|+|h_1| \geq 2 \right\} 
\end{align*}
might be infinite. \\
(4) We establish a new version of the Harnack inequality and derive
a-priori
Hölder regularity estimate as a consequence. In a different setting, this
procedure was recently established in \cite{Kas11b}. 

The constants in the main results of our work and \cite{BL1} depend on $\alpha$. It would be desirable to adopt the technique further such that results would be robust for $\alpha \to 2$. Under an assumption like
(\ref{eq:n_assum_bl}), this has been acheived with analytic techniques in
\cite{Sil06} and \cite{Kas11b}.

Comparing our results to the local
theory of second order partial differential equations, a natural question
arises: Which is a natural class of kernels $n$ such that similar results hold
true? 

We call a kernel $n$ of the above type nondegenerate if there is a
function $N:(0,1) \to (0,\infty)$ with $\lim\limits_{\rho \to 0+} N(\rho) =
+\infty$ and $\lambda, \Lambda >0$ such that for every $\rho \in (0,1)$ and $x
\in \R^d$ the symmetric matrix $[A^\rho_{ij}(x)]_{i,j=1}^d$ defined by 
\[ A^\rho_{ij}(x) = N(\rho) \int\limits_{\{0 < |h| \leq \rho\}} h_i h_j  n(x,h)
\, dh \,. \] 
satisfies for every $\xi \in \R^d$ 
\begin{align}\label{eq:nondeg_assum_elliptic}
\lambda |\xi|^2 \leq \sum\limits_{i,j=1}^d A^\rho_{i,j}(x) \xi_i \xi_j \leq \Lambda |\xi|^2 \,.
\end{align}

If $n$ depends only on $h$ and $N(\rho)=\rho^{\alpha-2}$, then this condition
implies that the corresponding L\'{e}vy process has a smooth density, see
\cite{Pic}. Note that condition (\ref{eq:n_assum_bl}) implies the nondegeneracy
condition (\ref{eq:nondeg_assum_elliptic}) with $N(\rho)=\rho^{\alpha-2}$ but is
not necessary, just consider the example 
\begin{align}\label{eq:n_cone} 
%n(x,h) = |h|^{-d-\alpha} \mathbbm{1}_{\{|h_1|\geq 99 |h|\}} \,.
%Changed Aug15 to
n(x,h) = |h|^{-d-\alpha} \mathbbm{1}_{\{|h_1|\geq 0.99 |h|\}} \,.
\end{align}
Note that (\ref{eq:nondeg_assum_elliptic}) holds under our assumptions.

Let us comment on other articles that generalize the results of
\cite{BL1}. Note that we do not include works on nonlocal Dirichlet forms. 
\cite{SV1} gives conditions on L\'evy processes and more general Markov jump
processes such that the theory of \cite{BL1} is applicable. In \cite{BKa} the
theory is extended to the variable order case and to situations where the lower
and upper bound in (\ref{eq:n_assum_bl}) behave differently for $|h| \to 0$. In
these cases, regularity of harmonic functions does not hold. Regularity is
established in \cite{BaKa05PDE} for variable order cases under additional
assumptions. Fine potential theoretic results are obtained in \cite{BSS02,
BoSt05} for stable processes. The case of L\'{e}vy processes with truncated
stable L\'{e}vy densities is covered in \cite{KS} and generalized in \cite{Mi}.
As mentioned above there is an independent approach with analytic methods
developed in \cite{Sil06, CaSi09} covering linear and fully nonlinear
integro-differential operators. 

{\bf Notation:} For two functions $f$ and $g$ we write $f(t) \sim g(t)$ if
$f(t)/g(t) \to 1$. For $A \subset \R^d$ open or closed $\tau_A$ denotes the
first exit time of the Markov process under consideration. $T_A$ denotes the
the first hitting time of the set $A$.

{\bf Acknowledgement:} The authors thank an anonymous referee for pointing out that the previous version of assumptions (\ref{eq:main_assum}), (\ref{eq:k}) was overly general. Example 2 was added in order to motivate these assumptions.

\section{Some probabilistic estimates}

In this section we prove useful auxiliary results. We follow closely
the ideas of \cite{BL1}. However, we need to provide several computations because
of the appearance of a slowly varying function in (J1). The proofs of
Proposition \ref{prop:ks} and Proposition \ref{prop:rhi} are significantly
different from their counterparts in \cite{BL1}.

The following proposition will be used often in obtaining probabilistic estimates.
\begin{Prop}\label{prop:comp}
Let $A,B\subset \R^d$ be disjoint Borel sets. Then for every bounded stopping
time $T$  
\[
\E^x\left[\sum_{s\leq T} \mathbbm{1}_{\{X_{s-}\in A, X_s \in B\}}\right]=\E^x\left[\int_0^T \int _B \mathbbm{1}_{A}(X_s)\,n(X_s,u-X_s)\, du\right]
\]
for every $x\in \R^d$.
\end{Prop}
\begin{proof}
By \cite[Proposition 2.3]{BL1} it follows that the process
\[
\left\{\sum_{s\leq t} \mathbbm{1}_{\{X_{s-}\in A, X_s \in B\}}-\int_0^t \int _B \mathbbm{1}_{A}(X_s)\,n(X_s,u-X_s)\, du\right\}_{t\geq 0}
\]
is a $\P^x$-martingale. Therefore the result follows by the optional stopping theorem.
\end{proof}

The following result, taken from the theory of regular variation, will be repeatedly used throughout the paper.
\begin{Prop}\label{prop:help}
Assume that $\ell\colon (0,2)\rightarrow (0,\infty)$ varies slowly at $0$ and let $\beta_1>-1$ and $\beta_2>1$. Then the following is true:
\begin{itemize}
	\item[(i)] $\ds{\int_0^r\,u^{\beta_1}\ell(u)\,du\sim \frac{r^{1+\beta_1}}{1+\beta_1}}\,\ell(r)$ as $ r\to 0+$,
	\item[(ii)] $\ds{\int_r^1\,u^{-\beta_2}\ell(u)\,du\sim \frac{r^{1-\beta_2}}{\beta_2-1}}\,\ell(r)$ as $ r\to 0+$.
\end{itemize}
\end{Prop}
\begin{proof}
By a change of variables and using \cite[Proposition 1.5.10]{BGT} we obtain 
\begin{align*}
	\int_0^r\,u^{\beta_1}\ell(u)\,du&=\int_{r^{-1}}^\infty u^{-\beta_1-2}\ell(u^{-1})\,du\sim \frac{r^{1+\beta_1}\ell(r)}{1+\beta_1},
\end{align*}
since $u\mapsto \ell(u^{-1})$ varies slowly at infinity. This proves (i). Similarly, with the help of \cite[Proposition 1.5.8]{BGT} we obtain (ii). 
\end{proof}

\begin{Rem}\label{rem:decr}
Using  \cite[Theorem 1.5.4]{BGT} we conclude that for a function $\ell\colon (0,2)\rightarrow (0,\infty)$ that varies slowly at $0$ there exists a non-increasing function $\phi\colon (0,2)\rightarrow (0,\infty)$ such that 
\begin{equation*}\label{eq:tmp122}
\lim_{r\to 0+}\frac{r^{-d-\alpha}\ell(r)}{\phi(r)}=1.
\end{equation*}
\end{Rem}

Before proving our main probabilistic estimates, note that (\ref{eq:k}) implies
that there exists $\vartheta\in (0,\pi/2]$ such that for every $i \in \{1, \ldots, N\}$ 
\begin{equation}\label{eq:cone1}
	 n(x,h)\geq \delta\,j(|h|)\ \ \textrm{ for all }\, h\in\R^d,\ h\not=0,\ \frac{| \langle h, \eta_i \rangle|}{|h|}\geq \cos\vartheta.
\end{equation}

\subsection{Exit time estimates}
\begin{Prop}\label{prop:estp} There exists a constant $C_1>0$ such that for
every $x_0\in \R^d$, $r\in (0,1)$ and $t>0$
\[
\P^{x_0}(\tau_{B(x_0,r)}\leq t)\leq C_1t\frac{\ell(r)}{r^\alpha}.
\]
\end{Prop}
\begin{proof} Again, we closely follow the ideas in \cite{BL1}. Let $x_0\in
\R^d$, $r\in (0,1)$ and let $f\in C^2(\R^d)$ be a positive function such that
\[
f(x)=\left\{\begin{array}{cl}
|x-x_0|, & |x-x_0|\leq \frac{r}{2}\\
r^2,& |x-x_0|\geq r
\end{array}\right.
\]
and
\[
|f(x)|\leq c_1 r^2, \ \ \left|\frac{\partial f}{\partial x_i}(x)\right|\leq c_1 r \ \ \textrm{ and } \ \ \left|\frac{\partial^2 f}{\partial x_i \partial x_j}(x)\right|\leq c_1,
\]
for some constant $c_1>0$.

Let $x\in B(x_0,r)$. We estimate $\mathcal{L}f(x)$ in a few steps.

First
\begin{align*}
&\ \ \ \int_{B(x_0,r)}\left(f(x+h)-f(x)-\langle\nabla f(x), h\rangle\mathbbm{1}_{\{|h|\leq 1\}}\right)n(x,h)\,dh\\& \leq c_2\int_{B(x_0,r)}|h|^2n(x,h)\,dh\leq c_2\ \int_{B(x_0,r)}|h|^{2-d-\alpha}\ell(|h|)\,dh\\
&\leq c_3 r^{2-\alpha}\ell(r),
\end{align*}
where in the last line we have used Proposition \ref{prop:help} (i).
Similarly, by Proposition \ref{prop:help} (ii) on $B(x_0,r)^c$ we get
\begin{align*}
&\ \ \  \int_{B(x_0,r)^c}\left(f(x+h)-f(x)\right)n(x,h)\,dh\leq \|f\|_\infty \int_{B(x_0,r)^c}n(x,h)\,dh\\
& \leq \|f\|_\infty \left(\int_{B(x_0,1)\setminus B(x_0,r)}|h|^{-d-\alpha}\ell(|h|)\,dh+\int_{B(x_0,1)^c}n(x,h)\,dh\right)\\
& \leq c_1r^2\left(c_4 r^{-\alpha}\ell(r)+c_5 \right)
\leq c_6r^{2-\alpha}\ell(r).
\end{align*}
In the last inequality we have used the fact that $\ds{\lim_{r\to
0+}}r^{-\alpha}\ell(r)=\infty$ (cf. \cite[Proposition 1.3.6 (v)]{BGT}). Finally,
by symmetry of the kernel, we have 
\begin{align}
\int_{B(x_0,1)\setminus B(x_0,r)} \langle h, \nabla f(x) \rangle n(x,h)\,dh=0.
\end{align}
Therefore, by preceding estimates, we conclude that there is a constant $c_7>0$
such that for all $x\in \R^d$ and $r \in (0,1)$ 
\begin{equation}\label{eq:prop1e3halbe}
\mathcal{L}f(x)\leq c_7 r^{2-\alpha} \ell(r).
\end{equation}
It follows from the optional stopping theorem that
\begin{equation}\label{eq:prop1e2}
\E^{x_0} f(X_{t\wedge \tau_{B(x_0,r)}})-f(x_0)=\E^{x_0} \int_0^{t\wedge \tau_{B(x_0,r)}}\mathcal{L}f(X_s)\, ds\leq c_7 t r^{2-\alpha} \ell(r),\ \ t>0.
\end{equation}
On $\{\tau_{B(x_0,r)}\leq  t\}$ one has $X_{t\wedge
\tau_{B(x_0,r)}}\not\in
B(x_0,r)$ and so $f(X_{t\wedge \tau_{B(x_0,r)}})\geq r^2$. Then  (\ref{eq:prop1e2}) gives
\[
\P^{x_0}(\tau_{B(x_0,r)}\leq t)\leq c_7 t\,r^{-\alpha}\ell(r).
\]
\end{proof} 
\begin{Prop}\label{prop:est2} There exists a constant $C_2>0$ such that for
every $r\in (0,1)$  and $ x_0\in \R^d$  
\[
\inf_{y\in B(x_0,r/2)}\E^y \tau_{B(x_0,r)}\geq C_2\frac{r^\alpha}{\ell(r)}.
\]
\end{Prop}
\begin{proof} Let $r\in (0,1)$, $x_0\in \R^d$ and $y\in B(x_0,r/2)$. Using
Proposition \ref{prop:estp} we obtain
\[
	\P^y(\tau_{B(x_0,r)}\leq t)\leq \P^y(\tau_{B(y,r/2)}\leq t)\leq C_1\,t\,r^{-\alpha}\ell(r)\ \textrm{ for }\ t>0.
\]
Let 
\[
t_0=\frac{r^{\alpha}}{2 C_1\ell(r)}.
\]
Then
\[
	 \E^y \tau_{B(x_0,r)}\geq t_0\P^y(\tau_{B(x_0,r)}\geq t_0)\geq \frac{r^{\alpha}}{2 C_1\ell(r)}.
\]
\end{proof}

\begin{Prop}\label{prop:est1} There exists a constant $C_3>0$ such that for
every $r\in (0,\frac{1}{2})$  and $ x_0\in \R^d$  
\[
\sup_{y\in B(x_0,r)}\E^y \tau_{B(x_0,r)}\leq C_3\frac{r^\alpha}{\ell(r)}.
\]
\end{Prop}
\begin{proof} Let $r\in (0,\frac 12)$, $x_0\in \R^d$ and $y\in B(x_0,r)$. 
Denote by  $S$ the first time when process $(X_t)_{t\geq 0}$ has a jump larger than $2r$, i.e.
\[
S=\inf\{t>0\colon |X_t-X_{t-}|> 2r\}.
\]
Assume first that $\P^y(S\leq \frac{r^{\alpha}}{\ell(r)})\leq \frac 12$. Then by
Proposition \ref{prop:comp} 
\begin{align}
	\P^y\left(S\leq \tfrac{r^{\alpha}}{\ell(r)}\right)&=\E^y
\left[\sum_{s\leq \frac{r^{\alpha}}{\ell(r)}\wedge
S}\mathbbm{1}_{\{|X_s-X_{s-}|>2r\}}\right]\nonumber\\
	&=\E^y \left[\int_0^{\frac{r^{\alpha}}{\ell(r)}\wedge S}\int_{B(0,2r)^c}n(X_s,h)\,dh\,ds\right]\label{eq:lowest1}
\end{align}

Choose arbitrary $\xi_0 \in \{\eta_1, \ldots, \eta_N\}$ and let $\vartheta$ be as in (\ref{eq:cone1}). Then
\begin{align*}
	\int_{B(0,2r)^c}n(X_s,h)\,dh&\geq \int_{\left\{h\in\R^d\colon 2r\leq |h|<1,\frac{| \langle h, \xi_0 \rangle |}{|h|}\geq \cos \vartheta\right\}}n(X_s,h)\,dh\\
	&\geq \delta\int_{\left\{h\in\R^d\colon 2r\leq |h|<1,\frac{| \langle h, \xi_0 \rangle |}{|h|}\geq \cos \vartheta\right\}} \frac{\ell(|h|)}{|h|^{d+\alpha}}\,dh\\
	&\geq c_1\int_{2r}^1\frac{\ell(t)}{t^{1+\alpha}}\,dt\geq c_2 \frac{\ell(r)}{r^\alpha},
\end{align*}
where in the last inequality we have used Proposition \ref{prop:help} (ii). Using this estimate we get from (\ref{eq:lowest1})  the following estimate
\begin{align*}
	\P^y\left(S\leq \frac{r^{\alpha}}{\ell(r)}\right)&\geq c_2\frac{\ell(r)}{r^{\alpha}}\E^y\left[\frac{r^{\alpha}}{\ell(r)}\wedge S\right]\\
	&\geq c_2\P^y\left(S>\frac{r^{\alpha}}{\ell(r)}\right)\geq \frac{c_2}{2}.
\end{align*}
Therefore, in any case the following inequality holds:
\[
\P^y\left(S\leq \frac{r^\alpha}{\ell(r)}\right)	\geq \frac{1}{2}\wedge \frac{c_2}{2}.
\]

 Since $S\geq \tau_{B(x_0,r)}$ we conclude
\[\P^y\left( \tau_{B(x_0,r)}\leq \frac{r^\alpha}{\ell(r)}\right)\geq \P^y\left(S\leq \frac{r^\alpha}{\ell(r)}\right)\geq c_3,\]
 with $c_3=\frac{1}{2}\wedge \frac{c_2}{2}$. By the Markov property, for $m\in
\N$ we obtain
 \begin{align*}
 \P^y\left(\tau_{B(x_0,r)}>(m+1)\frac{r^\alpha}{\ell(r)}\right)&\leq \P^y\left(\tau_{B(x_0,r)}>m\frac{r^\alpha}{\ell(r)},\tau_{B(x_0,r)}\circ \theta_{m\frac{r^\alpha}{\ell(r)}}>\frac{r^\alpha}{\ell(r)}\right)\\
 &=\E^y\left[\P^{X_{m\frac{r^\alpha}{\ell(r)}}}\left(\tau_{B(x_0,r)}>\frac{r^\alpha}{\ell(r)}\right);\tau_{B(x_0,r)}>m\frac{r^\alpha}{\ell(r)}\right]\\&\leq (1-c_3)\P^y\left(\tau_{B(x_0,r)}>m\frac{r^\alpha}{\ell(r)}\right),
 \end{align*}
 where $\theta_s$ denotes the usual shift operator.
By iteration we obtain
 \[
 	\P^y\left(\tau_{B(x_0,r)}>m\frac{r^\alpha}{\ell(r)}\right)\leq (1-c_3)^m,\ m\in \N.
 \]
 Finally,
\begin{align*}
\E^y \tau_{B(x_0,r)}&\leq  \frac{r^\alpha}{\ell(r)}\sum_{m=0}^\infty(m+1)\P^y\left(\tau_{B(x_0,r)}>m\frac{r^\alpha}{\ell(r)}\right)\\&\leq \frac{r^\alpha}{\ell(r)}\sum_{m=0}^\infty (m+1)(1-c_3)^m \leq c_4 \frac{r^\alpha}{\ell(r)}.
\end{align*}
\end{proof}

\subsection{Krylov-Safonov type estimate}

Fix $\vartheta \in (0,\pi/2]$ such that (\ref{eq:cone1}) holds.

\begin{Prop}\label{prop:ks}Let $\lambda \in
\left(0,\frac{\sin\vartheta}{8}\right]$. There exists a constant
$C_4=C_4(\lambda)>0$ such that for every $x_0\in\R^d$, $r\in (0,\frac 12)$,
closed
set $A\subset B(x_0,\lambda r)$
and $x\in B(x_0,\lambda r)$,
\[
\P^x(T_A<\tau_{B(x_0,r)})\geq C_4\frac{|A|}{|B(x_0,r)|}.
\]
\end{Prop}
\begin{proof}
Choose arbitrary $\xi_0 \in \{\eta_1, \ldots, \eta_N\}$ and set
$\tilde{x}_0=x_0-\frac{r}{2}\xi_0$. The idea is to choose $\lambda\in (0,\frac
18]$ such that 
\begin{equation}\label{eq:cone2}
	\frac{|\langle u-v , \xi_0 \rangle |}{|u-v|}\geq \cos\vartheta
\end{equation}
for all  $u\in B(x_0,2\lambda r),\ v\in B(\tilde{x}_0,2\lambda r)$.
Since for every $u\in B(x_0,2\lambda r)$ and $v\in B(\tilde{x}_0,2\lambda r)$ 
\[
	\frac{| \langle u-v, \xi_0\rangle|}{|u-v|}\geq \frac{\sqrt{(\frac r 4)^2-(2\lambda r)^2}}{\frac r 4}=\sqrt{1-(8\lambda)^2}.
\]
 it is enough to choose $\lambda\in (0,\frac 18]$ such that
\[
	\sqrt{1-(8\lambda)^2}\geq \cos\vartheta,
\]
or, more explicitly,
\[
	\lambda\leq  \frac{\sin\vartheta}{8}.
\]
{
\begin{figure}
 \centering 
\psset{unit=0.5cm}
\fontsize{11}{10}\selectfont
\begin{pspicture*}(-5,-8)(15,8)
\psset{linewidth=\pslinewidth}
\psset{dotsize=5pt 0}
\pstGeonode[PosAngle={60,60,90},PointSymbol=*,PointName=v](1.8,3){u}
\pstGeonode[PosAngle={0,60,90},PointSymbol=*](5,0){x_0}
\pstGeonode[PosAngle={-60,-135,90},PointSymbol=*,PointName=\widetilde{x}_0](1.40
,2){w}
\pstGeonode[PosAngle={0,-135,90},PointSymbol=none,PointName=none](12.5,0){q}
(3.80,0){t}(4.5,3.3){m}
\pstCircleOA[linecolor=black,linewidth=0.04,linestyle=dotted,
dotsep=0.05]{x_0}{q}
\pstCircleOA[linecolor=black,linewidth=0.04,linestyle=dotted,
dotsep=0.05]{x_0}{w}
\pstCircleOA[linecolor=black,linewidth=0.06]{x_0}{t}
\pstCircleOA[linecolor=black,linewidth=0.06,Radius=\pstDistAB{x_0}{t}]{w}{}
\pstTranslation[PointName=none,PointSymbol=none]{w}{u}{x_0}
\pstLineAB[linecolor=black,nodesepA=12,nodesepB=8,linestyle=dotted,dotsep=0.05,
linewidth=0.06]{u}{x_0'}
\pstOrtSym[PointName=none,PointSymbol=none]{u}{x_0'}{m}
\pstLineAB[linecolor=black,nodesepA=12,nodesepB=8]{u}{m}
\pstLineAB[linecolor=black,nodesepA=12,nodesepB=8]{u}{m'}
\pstMarkAngle[MarkAngleRadius=2.4,arrows=-,LabelSep=2]{x_0'}{u}{m}{
$\vartheta$}
\pstMarkAngle[MarkAngleRadius=2.4,arrows=-,LabelSep=2]{m'}{u}{x_0'}{
$\vartheta$}
\pstGeonode[PointName=none,PointSymbol=none](9.1,0.3){t1}(11,1.6){t2}
\ncline[linecolor=black,linestyle=dotted,
dotsep=0.06, arrowscale=2]{->}{t2}{t1}
\rput(10.6,2){$B(x_0,\frac{r}{2})$}
\end{pspicture*}
\caption{\label{fig:one} The choice of $\widetilde{x}_0$ and $\lambda$.}
\end{figure}
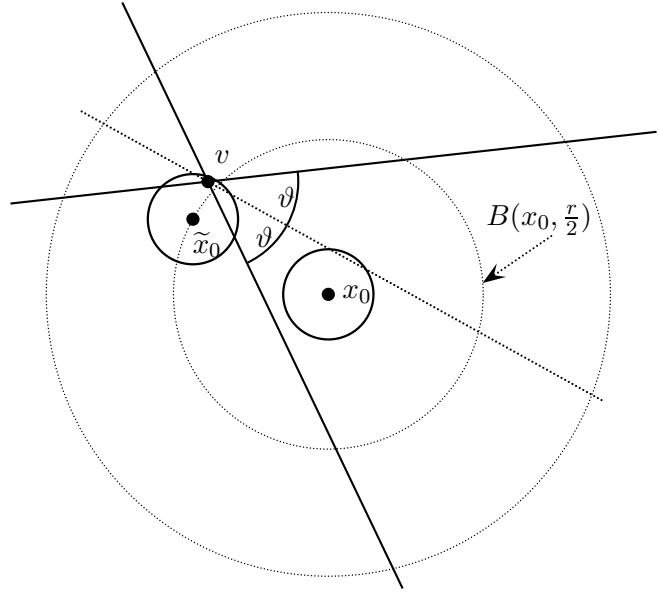

}
For $s>0$ we denote $B(x_0,s)$ and $B(\tilde{x}_0,s)$ by $B_s$ and
$\tilde{B}_s$. Let $r\in (0,1)$, $\lambda\in (0,\frac{\sin\vartheta}{8}]$, $x\in
B_{\lambda r}$ and let $A\subset B_{\lambda r}$ be a
closed subset. The strong Markov property now implies
\begin{align}
	\P^x(T_A<\tau_{B_r})& \geq \P^x\left(X_{\tau_{B_{2\lambda r}}}\in \tilde{B}_{\lambda r},X_{\tau_{\tilde{B}_{2\lambda r}}}\circ \theta_{\tau_{B_{2\lambda r}}}\in A\right)\nonumber\\
	&=\E^x\left[\P^{X_{\tau_{B_{2\lambda r}}}}(X_{\tau_{\tilde{B}_{2\lambda r}}}\in A);X_{\tau_{B_{2\lambda r}}}\in \tilde{B}_{\lambda r}\right].\label{eq:smp1}
\end{align}
For every $y\in \tilde{B}_{\lambda r}$ and $t>0$ Proposition \ref{prop:comp}
and  (\ref{eq:cone2}) yield
\begin{align*}
	\P^y&(X_{\tau_{\tilde{B}_{2\lambda r}}\wedge t}\in A)=
\E^y\left[\sum_{s\leq \tau_{\tilde{B}_{2\lambda r}}\wedge
t}\mathbbm{1}_{\{X_{s-}\not=X_s,X_s\in A\}}\right]\\
	&=\E^y\left[\int_0^{\tau_{\tilde{B}_{2\lambda r}}\wedge t}\int_A
n(X_s,z-X_s)\,dz\,ds\right] \geq
\delta\,\E^y\left[\int_0^{\tau_{\tilde{B}_{2\lambda r}}\wedge t}\int_A
\frac{\ell(|z-X_s|)}{|z-X_s|^{d+\alpha}}\,dz\,ds\right].
\end{align*}
Letting $t\to \infty$ and using the monotone convergence theorem we deduce
\begin{equation*}
\P^y(X_{\tau_{\tilde{B}_{2\lambda r}}}\in A)\geq \delta\,\E^y\left[\int_0^{\tau_{\tilde{B}_{2\lambda r}}}\int_A \frac{\ell(|z-X_s|)}{|z-X_s|^{d+\alpha}}\,dz\,ds\right].
\end{equation*}
Since $|z-X_s|\leq r/2+4\lambda r\leq r$, by Remark \ref{rem:decr} we conclude
\begin{align*}
	\P^y(X_{\tau_{\tilde{B}_{2\lambda r}}}\in A)&\geq c_1\frac{\ell(r)}{r^{d+\alpha}}|A|\E^y\tau_{\tilde{B}_{2\lambda r}}\\
	&\geq c_2 \ell(r)\frac{|A|}{|B_r|}r^{-\alpha} \E^y\tau_{\tilde{B}_{2\lambda r}}.
\end{align*}
Using  Proposition \ref{prop:est2} we deduce
\begin{equation}\label{eq:tmphh100}
	\P^y(X_{\tau_{\tilde{B}_{2\lambda r}}}\in A)\geq c_3 \frac{\ell(r)}{\ell(2\lambda r)}\lambda^\alpha\frac{|A|}{|B_r|}.
\end{equation}
%It follows from   (\ref{eq:lcond}) and Potter's theorem (cf. \cite[Theorem 1.5.6 (i)]{BGT}) that 
%\begin{align}
%	\frac{\ell(r)}{\ell(2\lambda r)}&\geq c_4\frac{\ell(r)}{\ell(2\lambda r)}\nonumber \\&\geq c_5 \left(\frac{r}{2\lambda r}\right)^{\frac{\alpha}{2}}
%		=c_6 \lambda^{-\frac{\alpha}{2}},\label{eq:potter}
%\end{align}
%which combined  with (\ref{eq:tmphh100}) gives 
%\[
%	\P^y(X_{\tau_{B\left(\tilde{x}_0,2\lambda r\right)}}\in A)\geq c_7 \lambda^{\frac{\alpha}{2}}\frac{|A|}{|B(x_0,r)|}.
%\]
Since $\ell$ varies slowly at $0$ we finally obtain
\begin{equation}\label{eq:ks1}
	\P^y(X_{\tau_{\tilde{B}_{2\lambda r}}}\in A)\geq c_4 \frac{|A|}{|B_r|}\ \ \textrm{ for all }\ y\in \tilde{B}_{\lambda r},
\end{equation}
for some constant $c_4=c_4(\lambda)>0$.
By symmetry and (\ref{eq:ks1}) we deduce
\begin{equation}\label{eq:ks2}
	\P^x(X_{\tau_{B_{2\lambda r}}}\in \tilde{B}_{\lambda r})\geq c_4 \frac{|\tilde{B}_{\lambda r}|}{|\tilde{B}_r|}\ \textrm{ for all }\ x\in B_{\lambda r}.
\end{equation}
Finally, by (\ref{eq:smp1}), (\ref{eq:ks1}) and (\ref{eq:ks2}) we get
\[
	\P^x(T_A<\tau_{B_r})\geq c_4^2 \lambda^{d}\frac{|A|}{|B_r|}.
\]
\end{proof}

\subsection{Restricted Harnack inequality} The aim of this subsection is to establish a Harnack inequality for a restricted class of harmonic functions.  

The following lemma can be proved similarly as \cite[Lemma 2.7]{Mi}.
\begin{Lem}\label{lem:ujs}
	Let $g\colon (0,\infty)\rightarrow [0,\infty)$ be a function  satisfying
	\[
		g(s)\leq c g(t)\ \textrm{ for all }\ 0<t\leq s,
	\]
	for some constant $c>0$. There is a constant $c'>0$ such that for any $x_0\in \R^d$ and $r>0$  we have 
	\[
		g(|z-x|)\leq c' r^{-d}\int_{B(x_0,r)}g(|z-u|)\,du,
	\]
	for all $x\in B(x_0,r/2)$ and $z\in B(x_0,2r)^c$.
\end{Lem}

%Recall that the choice of $\vartheta \in (0,\pi/2]$ assures that (\ref{eq:cone1}) holds.
%\begin{Prop}\label{prop:rhi} Let $\lambda \in \left(0,\frac{\tan{(\vartheta/2)}\wedge \sin\vartheta}{4}\right]$. 
%	There exists a constant $C_5=C_5(\lambda)\geq 1$ such that for all $x_0\in\R^d$, $r\in (0,1/2)$ and  $x,y\in B(x_0,\lambda r)$ 
%	\[
%		\E^x[H(X_{\tau_{B(x_0,\lambda r)}})]\leq C_5 \E^y[H(X_{\tau_{B(x_0,r)}})],
%	\]
%	for any non-negative function $H\colon \R^d\rightarrow [0,\infty)$
%supported in $B(x_0,3r/2)^c$.
%\end{Prop}
\begin{Prop}\label{prop:rhi} There is a constant $\lambda_0\in (0,\frac{1}{16})$
so that for every $\lambda\in (0,\lambda_0]$ 
	there exists a constant $C_5=C_5(\lambda)\geq 1$ such that for all
$x_0\in\R^d$, $r\in (0,\frac{1}{2})$ and  $x,y\in B(x_0,\lambda r)$ 
	\[
		\E^x[H(X_{\tau_{B(x_0,\lambda r)}})]\leq C_5
\E^y[H(X_{\tau_{B(x_0,r)}})],
	\]
	for every non-negative function $H\colon \R^d\rightarrow [0,\infty)$
supported in $B(x_0,3r/2)^c$.
\end{Prop}
\begin{proof}
	Let $x_0\in\R^d$, $r\in (0,\frac{1}{2})$ and let $x,y\in B(x_0,\lambda
r)$, where $\lambda\in (0,\lambda_0)$ and $\lambda_0 \in (0,\frac{1}{16})$ is
chosen later. $\lambda_0$ will depend only on constants in our main assumptions.
 Take $z\in B(x_0,3r/2)^c$. There are only two cases. 

	{\bf Case 1:} There exists $u_0\in B(x_0,\lambda r)$ so that
$n(u_0,z-u_0)>0$.

	{\bf Case 2:} $n(u,z-u)=0$ for all $u\in B(x_0,\lambda r)$.

        We consider Case 1.  By (\ref{eq:main_assum}) and (\ref{eq:k}) there exist $\xi' \in \{\pm \eta_1, \ldots, \pm \eta_N\}$ and $\vartheta' \in (0,\frac{\pi}{2}]$ with 
	\[
		\frac{\langle z-u_0, \xi' \rangle}{|z-u_0|}\geq
\cos\vartheta'.
	\]
	Note that $\xi', \vartheta'$ depend on $u_0, z, x_0$ and $r$ but $\vartheta' \geq \vartheta$ uniformly with $\vartheta$ as in (\ref{eq:cone1}). 

Set $\tilde{x_0}=x_0-\frac{r}{2}\xi'$ and take $\lambda_0\leq \frac{\sin
\vartheta}{16}$. Let $B_s:=B(x_0,s)$ and
$\tilde{B}_s:=B(\tilde{x}_0,s)$. As in (\ref{eq:cone2}), for $\lambda\leq \lambda_0$ we have
\[
 \frac{|\langle u-v,\xi'\rangle|}{|u-v|}\geq \cos{\xi'} \ \text{ for all } \
u\in B_{2\lambda r},\  v\in\tilde{B}_{2\lambda r}\,. 
\]

Choose $\tilde{z_0}\in \partial B_{r/2}$ so that the following conditions hold:
\begin{equation}\label{eq:rhi_tmp1}
\begin{aligned}{2}
%\begin{split}
 |z-w|&\leq |z-u|\ &&\text{ for all }\ u\in B_{2\lambda r},\ w\in
B(\tilde{z_0},\tfrac{\lambda r}{4})  \,, \\
 \frac{\langle w-v,\xi'\rangle}{|w-v|}&\geq \cos{\vartheta'} \ &&\text{ for all } \
v\in \tilde{B}_{2\lambda r},\  w\in B(\tilde{z_0},\tfrac{\lambda r}{4}) \,, \\ 
\frac{\langle z-w,\xi'\rangle}{|z-w|}&\geq \cos{\vartheta'} \ &&\text{ for all } \ w\in B(\tilde{z_0},\tfrac{\lambda r}{4}) \,. 
%\end{split}
\end{aligned}
\end{equation}
In the appendix we briefly explain the geometric argument behind the choice of $\tilde{z_0}\in \partial B_{r/2}$. 

Let $B'_s=B(\tilde{z_0},s)$. By the strong Markov property,
\begin{align}
 \E^y\left[\int_0^{\tau_{B_r}}n(X_s,z-X_s)\,ds\right]&\geq \E^y\left[\int_{\tau_{B_{2\lambda r}}}^{\tau_{B_r}}n(X_s,z-X_s)\,ds; X_{\tau_{B_{2\lambda r}}}\in \tilde{B}_{\lambda r}\right]\nonumber\\
 &=\E^y\left[\left\{\int_0^{\tau_{B_r}}n(X_s,z-X_s)\,ds\right\}\circ \theta_{\tau_{B_{2\lambda r}}};X_{\tau_{B_{2\lambda r}}}\in \tilde{B}_{\lambda r}\right]\nonumber\\
 &=\E^y\left[\E^{X_{\tau_{B_{2\lambda r}}}}\left[\int_0^{\tau_{B_r}}n(X_s,z-X_s)\,ds\right];X_{\tau_{B_{2\lambda r}}}\in \tilde{B}_{\lambda r}\right]\,.\label{eq:rhi_tmp4}
\end{align}

Similarly, for $v\in \tilde{B}_{\lambda r}$ we have
\begin{equation}\label{eq:rhi_tmp5}
	\E^v\left[\int_0^{\tau_{B_r}}n(X_s,z-X_s)\,ds\right]\geq \E^v\left[\E^{X_{\tau_{\tilde{B}_{2\lambda r}}}}\left[\int_0^{\tau_{B_r}}n(X_s,z-X_s)\,ds\right]; X_{\tau_{\tilde{B}_{2\lambda r}}}\in B'_{\frac{\lambda r}{8}}\right]\,.
\end{equation}

Let $w\in B'_{\frac{\lambda r}{8}}$. Then (J1), (J2), Proposition \ref{prop:est2} and (\ref{eq:rhi_tmp1})  yield
\begin{align}
	\E^w &\left[\int_0^{\tau_{B_r}}n(X_s,z-X_s)\,ds\right] \geq \E^w\left[\int_0^{\tau_{B'_{\frac{\lambda r}{4}}}}n(X_s,z-X_s)\,ds\right]\nonumber\\&\geq c_1 \E^w \left[\int_0^{\tau_{B'_{\frac{\lambda r}{4}}}}j(|z-X_s|)\,ds\right] \geq c_2\E^w\tau_{B'_{\frac{\lambda r}{4}}}(4\lambda r)^{-d}\int_{B_{4\lambda r}}j(|z-u|)\,du\nonumber\\
	&\geq c_3 \lambda^{\alpha-d}\frac{r^{\alpha-d}}{\ell(\frac{\lambda r}{4})}\int_{B_{4\lambda r}}j(|z-u|)\,du\,.\label{eq:rhi_tmp6}
\end{align}

Combining (\ref{eq:rhi_tmp4}),  (\ref{eq:rhi_tmp5}) and (\ref{eq:rhi_tmp6}) we obtain
\begin{align*}
\begin{split}
	\E^y&\left[\int_0^{\tau_{B_r}}n(X_s,z-X_s)\,ds\right]\\&\geq c_3\lambda^{\alpha-d}\frac{r^{\alpha-d}}{\ell(\frac{\lambda r}{4})}\int_{B_{4\lambda r}}j(|z-u|)\,du\, \E^y\left[\P^{X_{\tau_{B_{2\lambda r}}}}(X_{\tau_{\tilde{B}_{2\lambda r}}}\in B'_{\frac{\lambda r}{8}}); X_{\tau_{B_{2\lambda r}}}\in \tilde{B}_{\lambda r}\right]\,.
\end{split}
\end{align*}

Similarly as in the proof of Proposition \ref{prop:ks} we obtain, for some $c_4 = c_4(\lambda) > 0$
\[
	\P^v(X_{\tau_{\tilde{B}_{2\lambda r}}}\in B'_{\frac{\lambda r}{8}})\geq c_4 \ \text{ for all } \ v\in \tilde{B}_{\lambda r}
\]
and
\[
	\P^u(X_{\tau_{B_{2\lambda r}}}\in \tilde{B}_{\lambda r})\geq c_4 \ \text{ for all } \ u\in B_{\lambda r}\,.
\]

Therefore,
\begin{equation}\label{eq:tmprhi301}
	\E^y\left[\int_0^{\tau_{B_r}}n(X_s,z-X_s)\,ds\right]\geq c_5 \frac{r^{\alpha-d}}{\ell(\frac{\lambda r}{4})}\int_{B_{4\lambda r}}j(|z-u|)\,du\,.
\end{equation}

On the other hand, by Proposition \ref{prop:est1} and Lemma \ref{lem:ujs},
\begin{align}
	\E^x\left[\int_0^{\tau_{B_{\lambda r}}}n(X_s,z-X_s)\,ds\right]&\leq c_6\E^x\left[\int_0^{\tau_{B_{\lambda r}}}j(|z-X_s|)\,ds\right]\nonumber\\
		&\leq c_7 \E^x\tau_{B_{\lambda r}}(4 r)^{-d}\int_{B_{4\lambda r}}j(|z-u|)\,du\nonumber\\
		&\leq c_8 \frac{r^{\alpha-d}}{\ell(2\lambda r)}\int_{B_{4\lambda r}}j(|z-u|)\,du.\label{eq:tmprhi100}
\end{align}
It follows from  (\ref{eq:tmprhi301}) and (\ref{eq:tmprhi100}) that
\begin{equation}\label{eq:tmpwhi113}
	\E^x\left[\int_0^{\tau_{B_{\lambda r}}}n(X_s,z-X_s)\,ds\right]\leq c_9 \E^y\left[\int_0^{\tau_{B_r}}n(X_s,z-X_s)\,ds\right].
\end{equation}

Next, we consider Case 2, i.e. 	$n(u,z-u)=0$ for all $u\in B(x_0,\lambda r)$. Also in this case, assertion (\ref{eq:tmpwhi113}) holds true, because
\begin{align}\label{eq:tmprhi1}
 \begin{split}
\E^y\left[\int_0^{\tau_{B_r}}n(X_s,z-X_s)\,ds\right] &\geq 0 \,, \\
\E^x\left[\int_0^{\tau_{B_{\lambda r}}}n(X_s,z-X_s)\,ds\right] &= 0 \,.  
 \end{split}
\end{align}
We have shown that (\ref{eq:tmpwhi113}) always holds. It is enough to prove the proposition for $H=\mathbbm{1}_A$, where $A\subset B(x_0,3r/2)^c$. We conclude from  Proposition \ref{prop:comp} and (\ref{eq:tmpwhi113}) that
	\begin{align*}
		\P^y(X_{\tau_{B_r}}\in A)&=\int_{A}\E^y\left[\int_0^{\tau_{B_r}}n(X_s,z-X_s)\,ds\right]\,dz\nonumber\\
		&\geq c_9^{-1} \int_A\E^x\left[\int_0^{\tau_{B_{\lambda r}}}n(X_s,z-X_s)\,ds\right]\,dz	\\
		&=c_9^{-1} \P^x(X_{\tau_{B_{\lambda r}}}\in A).
	\end{align*}
\end{proof}

\pagebreak[4]
\section{Harnack inequality}

In this section we prove Theorem \ref{thm:harnack}.

\begin{proof}[Proof of Theorem \ref{thm:harnack}] Since $f$ is non-negative in $B(x_0,4r)$, we may assume that $\ds{\inf_{x\in
B(x_0,r)}f(x)}$ is positive. If not, we would prove the claim for  $f_\varepsilon
=f+\varepsilon$ and then consider $\varepsilon\to 0+$. By taking a constant
multiple of $f$ we may further assume $\ds{\inf_{x\in
B(x_0,r)}f(x)=\tfrac{1}{2}}$. 

Choose $u\in B(x_0,r)$ such that $f(u)\leq 1$. By Proposition
\ref{prop:est1} and using properties of slowly varying functions we can find a constant $c_1>0$ such that for all $u,v\in\R^d$ and $s\in (0,r]$
\begin{equation}\label{eq:tmp10}
	\E^u\tau_{B(v,2s)}\leq c_1\frac{s^\alpha}{\ell(s)}\ \text{ and }
\ \E^u\tau_{B(v,s)}\leq c_1\frac{r^\alpha}{\ell(r)} .
\end{equation}
From Proposition \ref{prop:ks} we deduce that there is a constant $c_2>0$ and $\lambda\in (0,\frac{\sin\vartheta}{16}]$ such that for all $A\subset B(x_0,2\lambda r)$ and $y\in B(x_0,2\lambda r)$
\begin{equation}\label{eq:ks}
\P^y(T_A<\tau_{B(x_0,2r)})\geq c_2\frac{|A|}{|B(x_0,2r)|}.
\end{equation}
Similarly, by  Proposition \ref{prop:ks} we see that there exists a constant
$c_3\in (0,1)$ such that for every $x\in \R^d$, $s<r$ and $C\subset B(x,\lambda
s)$ with $|C|/|B(x,\lambda s)|\geq \frac{1}{3}$
\[
	\P^x(T_C<\tau_{B(x,s)})\geq c_3.
\]
The idea of the proof is to show that $f$ is bounded from the above in
$B(x_0,r)$ by
\[
	c_4\left(1+\frac{r^\alpha}{\ell(r)}\sup_{v\in B(x_0,2r)}\int_{B(x_0,4r)^c}f^{-}(z)n(v,z-v)\,dz\right),
\]
for some constant $c_4>0$ that does not depend on $f$.
This will be proved by contradiction. 

Define
\begin{equation}\label{eq:eta}
	\eta=\frac{c_3}{3}\ \ \textrm{ and } \  \ \zeta=\frac{\eta}{2C_5},
\end{equation}
where $C_5$ is taken from Proposition \ref{prop:rhi}.

Assume that there exists $x\in B(x_0,\tfrac{3r}{2})$ such that $f(x)=K$ for
some 
\[
K>\max\left\{\frac{K_0}{\zeta},\frac{2\cdot 8^d\lambda^{-d}K_0}{c_2\zeta}\right\},
\]
where
\begin{equation}\label{eq:k0}
K_0=1+c_1\frac{r^\alpha}{\ell(r)}\sup_{v\in B(x_0,2r)}\int_{B(x_0,4r)^c}f^{-}(z)n(v,z-v)\,dz.
\end{equation}
Let $s=\left(\tfrac{2K_0}{c_2\zeta K}\right)^{1/d}2\lambda^{-1}\,r$. Then
$s<\frac{r}{4}$ and 
\[
|B(x,\lambda s)|=\frac{2K_0}{c_2\zeta\,K}|B(x_0,2r)| \,.
\]

Set $B_s:=B(x,s)$ and $\tau_s:=\tau_{B(x,s)}$. Let $A$ be a compact subset of
\[
A'=\{w\in B(x,\lambda s)\colon f(w)\geq \zeta K\}.
\]
By the optional stopping theorem, (\ref{eq:tmp10}), (\ref{eq:ks}) and  
Proposition \ref{prop:comp} 
\begin{align*}
1&\geq f(u)=\E^{u}[f(X_{T_A\wedge \tau_{B(x_0,2r)}})]\\
&\geq \E^{u}[f(X_{T_A\wedge \tau_{B(x_0,2r)}}); T_A<
\tau_{B(x_0,2r)}]-\E^{u}[f^-(X_{T_A\wedge \tau_{B(x_0,2r)}}); T_A>
\tau_{B(x_0,2r)}]\\
&\geq \zeta K\,\P^{u}(T_A< \tau_{B(x_0,2r)})-\E^{u}[f^-(X_{\tau_{B(x_0,2r)}})]\\
&=\zeta K\,\P^{u}(T_A<
\tau_{B(x_0,2r)})-\E^{u}\left[\int_0^{\tau_{B(x_0,2r)}}\int_{B(x_0,4r)^c}
f^-(z)n(X_t,z-X_t)\,dz\,dt\right]\\
&\geq c_2\,\zeta K\frac{|A|}{|B(x_0,2r)|}-c_1\tfrac{r^\alpha}{\ell(r)}\sup_{v\in
B(x_0,2r)}\,\int_{B(x_0,4r)^c}f^-(z)n(v,z-v)\,dz.
\end{align*}
Using (\ref{eq:k0}) we obtain
\begin{align*}
&\ \ \ \ \  \frac{|A|}{|B(x,\lambda s)|}\leq\\&\leq
\left(1+c_1\frac{r^\alpha}{\ell(r)}\sup_{v\in
B(x_0,2r)}\int_{B(x_0,4r)^c}f^-(z)n(v,z-v)\,dz\right)
\frac{|B(x_0,2r)|}{c_2\zeta K |B(x,\lambda s)|}\\&=\frac{K_0}{c_2\zeta
K}\frac{|B(x_0,2r)|}{|B(x,\lambda s)|}=\tfrac{1}{2},
\end{align*}
which implies
\[
	\frac{|A'|}{|B(x,\lambda s)|}\leq \frac{1}{2}.
\]
Let $C\subset B(x,\lambda s)\setminus A'$ be a compact subset such that
\begin{equation}\label{eq:defc}
	\frac{|C|}{|B(x,\lambda s)|}\geq \frac{1}{3}.
\end{equation}

Let $H=f^+\,\mathbbm{1}_{B_{3s/2}^c}$. Assume that
\begin{equation}\label{eq:tmpas}
\E^x[H(X_{\tau_{\lambda s}})]>\eta K.
\end{equation}
Then for any $y\in B(x,\lambda s)$ we have
\begin{align*}
f(y)&=\E^yf(X_{\tau_s})=\E^yf^+(X_{\tau_s})-\E^yf^-(X_{\tau_s})\nonumber\\
&=\E^yf^+(X_{\tau_s})-\E^y[f^-(X_{\tau_s});X_{\tau_s}\not\in B(x_0,4r)]\nonumber\\
&\geq \E^y[f^+(X_{\tau_s});X_{\tau_s}\not\in
B_{3s/2}]-\E^y[f^-(X_{\tau_s});X_{\tau_s}\not\in B(x_0,4r)].\label{eq:tmp1}
\end{align*}
Applying Proposition \ref{prop:rhi} to $H$ it follows
\begin{align*}
f(y)\geq C_5^{-1}&\E^x[f^+(X_{\tau_{\lambda s}});X_{\tau_{\lambda s}}\not\in B_{3s/2}]\\&-c_1\frac{r^\alpha}{\ell(r)}\sup_{v\in B(x_0,2r)}\int_{B(x_0,4r)^c}f^-(z)n(v,z-v)\,dz.
\end{align*}
Combining the last display with the assumption (\ref{eq:tmpas}) and the definition of $\zeta $ in (\ref{eq:eta}) gives
\[
	f(y)\geq C_5^{-1}\eta K-K_0=\zeta K\left(2-\tfrac{K_0}{\zeta
K}\right)\geq \zeta K \ \textrm{ for all }\ y\in B(x,\lambda s),
\]
which is a contradiction to (\ref{eq:defc}). Therefore
$\E^x[H(X_{\tau_{\lambda s}})]\geq \eta K$ . 

Let $M=\sup_{v\in B_{3s/2}}f(v)$. Then
\begin{align*}
	K=f(x)&=\E^x[f(X_{T_C});T_C<\tau_s]+\E^x[f(X_{\tau_s});\tau_s<T_C,X_{\tau_s}\in B_{3s/2}]\\
	&\ \ +\E^x[f(X_{\tau_s});\tau_s<T_C,X_{\tau_s}\not\in B_{3s/2}]\\
	&\leq \zeta K\,\P^x(T_C<\tau_s)+M(1-\P^x(T_C<\tau_s))+\eta K
\end{align*}
and thus
\[
	\frac{M}{K}\geq \frac{1-\eta -\zeta\,\P^x(T_C<\tau_s) }{1-\P^x(T_C<\tau_s)}.
\]
From the last display we conclude that $M\geq
K(1+2\beta)$ with 
$\beta=\tfrac{c_3}{6(1-c_3)}+\tfrac{\zeta}{2}>0$. Thus there
exists $x'\in B(x,\frac{3s}{2})$ so that $f(x')\geq K(1+\beta)$.

Using this procedure we obtain sequences $(x_n)$ and
$(s_n)$ such that $x_{n+1}\in
B(x_n,\frac{3s_n}{2})$ and $K_n:=f(x_n)\geq (1+\beta)^{n-1}K$. Thus
\[
	\sum_{n=1}^\infty |x_{n+1}-x_i|\leq \tfrac{3}{2}\sum_{n=1}^\infty
s_i\leq c_5\left(\tfrac{K_0}{K}\right)^{1/d}r,
\] 
for some constant $c_5>0$.

If $K>K_0\,c_5^d$, then $(x_n)$ is a sequence in $B(x_0,\frac{3r}{2})$  such
that 
\[
\lim_{n\to+\infty} f(x_n)\geq \lim_{n\to+\infty} (1+\beta)^{n-1}K_1=\infty.
\]
  This is a contradiction with the boundedness of  $f$ and so $K\leq c_5^dK_0$.
Thus
\begin{align*}
	\sup_{v\in B(x_0,r)}f(v)&\leq c_5^d\,K_0\\&=
c_5^d\left(1+\frac{r^\alpha}{\ell(r)}\sup_{v\in
B(x_0,2r)}\int_{B(x_0,4r)^c}f^-(z)n(v,z-v)\,dz\right) .
\end{align*}

Now, let $x,y\in B(x_0,r)$. Then
\begin{align*}
f(x)&\leq c_5^d\left(1+\frac{r^\alpha}{\ell(r)}\sup_{v\in
B(x_0,2r)}\int_{B(x_0,4r)^c}f^-(z)n(v,z-v)\,dz\right)\\&\leq
2c_5^df(y)+c_5^d\frac{r^\alpha}{\ell(r)}\sup_{v\in
B(x_0,2r)}\int_{B(x_0,4r)^c}f^-(z)n(v,z-v)\,dz.
\end{align*}
The proof is complete. 
\end{proof}

\section{Regularity estimates}

In this section we prove a general tool that
allows to deduce regularity estimates from the version of the Harnack equality
given in Theorem \ref{thm:harnack}. This approach is developed
in \cite{Kas11b}, see also Theorem 3 in \cite{KaDy11}. 

\begin{Thm}\label{thm:app_moritz} Let
$m\colon \R^d\times
\left(\R^d\setminus \{0\}\right)\rightarrow [0,\infty)$ be a measurable
function such that \mbox{$\sup\limits_{x\in\R^d} \int_{\R^d}
(|h|^2 \wedge 1) m(x,h) \, dh$} is finite. Assume there is a
function $\gamma:(0,\infty) \to (0,\infty)$ such that for all $x,h \in \R^d, h
\ne 0$ 
\begin{align}\label{eq:assum_badman}
 k\left(\tfrac{h}{|h|}\right) \gamma(|h|) \leq m(x,h) \leq \gamma(|h|)\,, 
\end{align}
where $k:S^{d-1} \rightarrow [0,\infty)$ is a measurable bounded symmetric function such that there is $\delta >0$ and a non-empty open set $I \subset S^{d-1}$ with $k(\xi) \geq \delta$ for every $\xi \in I$. Furthermore, assume that 
\begin{align}\label{eq:assum_gamma}
\limsup\limits_{R\to \infty} R^{\sigma_1} \int\limits_{B(0,R)^c} \gamma(|u|) \, du \leq 1 \,, \qquad \liminf\limits_{r\to 0+} r^{\sigma_2} \int\limits_{B(0,r)^c} \gamma(|u|) \, du \geq 1 \,,
\end{align}
with $0 < \sigma_1 \leq \sigma_2$. Let $\mathcal{L}$ be a non-local operator
defined by
\begin{align}\label{eq:defLm}
	\mathcal{L}f(x)=\int_{\R^d\setminus\{0\}}(f(x+h)-f(x)-\langle
\nabla f(x), h \rangle \,\mathbbm{1}_{\{|h|\leq 1\}})m(x,h)\,dh
\end{align}
for $f \in C_b^2(\R^d)$. \\
Assume that harmonic functions with respect to $\mathcal{L}$ satisfy a
Harnack inequality, i.e. there exist constants $c_1, c_2 \geq 1$ such that for
every
$x_0\in \R^d$, $r\in
(0,\frac{1}{4})$ and for every bounded function $f\colon \R^d\rightarrow \R$
which is
non-negative in $B(x_0,4r)$  and harmonic in $B(x_0,4r)$ the following
Harnack inequality holds for all $x,y\in B(x_0,r)$
\begin{align}\label{eq:assum_harnack}
f(x)\leq c_1 f(y)+ c_2 M(x_0, r) \sup_{v\in
B(x_0,2r)}\int_{B(x_0,4r)^c}f^-(z) m(v,z-v)\,dz \,, 
\end{align}
where $M(x_0,r)= (\int_{B(x_0,4r)^c} m(x_0, z-x_0) \, dz)^{-1}$. \\
Then there exist $\beta \in (0,1)$, $c \geq
1$
such that for every $x_0
\in \R^d$, every $R  \in (0,1)$, every function
$f:\R^d \to \R$ which is harmonic in $B(x_0,R)$ and every $\rho \in (0,R/2)$
\begin{align}
\sup\limits_{x,y \in B(x_0,\rho)} | f(x) - f(y) | &\leq c \|f\|_\infty
(\rho/R)^{\beta} \,. 
\end{align}
\end{Thm}

{\bf Remark:} Conditions (\ref{eq:assum_badman}), (\ref{eq:assum_gamma}), (\ref{eq:defLm}) do not imply in general that $\mathcal{L}$ satisfies a
Harnack inequality, see the discussion of Example 2.

Let us illustrate this result by giving two examples.

{\bf Example 3:} $m(x,h)= |h|^{-d-\alpha}$, i.e. $k\equiv 1$, $\gamma(t) =
t^{-d-\alpha}$, $\sigma_1=\sigma_2=\alpha$. Then $\mathcal{L} = c(\alpha)
\Delta^{\alpha/2}$. The Harnack inequality (\ref{eq:assum_harnack}) then becomes
\begin{align}\label{eq:assum_harnack_fracalpha}
f(x)\leq c_1 f(y)+ c_2 r^\alpha \int_{B(x_0,4r)^c}f^-(z)
|z-x_0|^{-d-\alpha} \,dz \,, 
\end{align}
and the theorem can be applied. Note that the function $f$ in
(\ref{eq:assum_harnack_fracalpha}) might be negative outside of $B(x_0,4r)$.

{\bf Example 4:} $m(x,h) \asymp |h|^{-d-\alpha}$, i.e. $k\equiv 1$,
$\gamma(t) = t^{-d-\alpha}$, $\sigma_1=\sigma_2=\alpha$, cf. \cite{BL1}. The
Harnack inequality can be formulated as in (\ref{eq:assum_harnack_fracalpha}).

\begin{proof}[Proof of Theorem \ref{thm:reg_est}] We apply Theorem
\ref{thm:app_moritz}. Let $k=k_1$ as in (\ref{eq:main_assum}) and $I=B_1$ as in (\ref{eq:k}). Set $m(x,h)=n(x,h)$, $\gamma(t)=j(t)$, $\sigma_1 = \sigma$ and $\sigma_2 = \alpha-\varepsilon$ where $\varepsilon \in
(0,\alpha-\sigma)$ is arbitrary. Then the first condition in
(\ref{eq:assum_gamma}) follows from (J3). The second condition follows from 
\[ r^{\sigma_2} \int_r^\infty
s^{d-1} j(s) \, ds = r^{\alpha-\varepsilon} \int_r^\infty
s^{-1-\alpha} \ell(s) \, ds \sim (1/\alpha) r^{-\varepsilon} \ell(r) \to +\infty
\text{ for } r \to 0+ \,, \]
where we use Proposition \ref{prop:help} (ii). It remains to check that there
is a constant $c>0$ such that for every $x_0 \in \R^d$ and every $r
\in (0,\frac 14)$ 
\begin{align*}
\frac{r^\alpha}{\ell(r)} \leq c M(x_0, r),  \quad \text{ i.e. } \; 
\int_{B(x_0,4r)^c} m(x_0, z-x_0) \, dz \leq c
\frac{\ell(r)}{r^\alpha} \,. 
\end{align*}
This condition follows from 
\begin{align}\label{eq:heart}
\int_{B(x_0,4r)^c} m(x_0, z-x_0) \, dz \leq \int_{B(x_0,4r)^c} j(|z-x_0|) \, dz
\leq c_3 \frac{\ell(4r)}{(4r)^\alpha} \leq c_4 \frac{\ell(r)}{r^\alpha} \,, 
\end{align} 
where we use Proposition \ref{prop:help} (ii) again.
\end{proof}

\begin{proof}[Proof of Theorem \ref{thm:app_moritz}] For $x_0\in \R^d$ and $r\in
(0,1)$ let $\nu_r^x$ denote the  measure on $B(x_0,r)^c$ defined by
\begin{align*}
 \nu_r^x(A) = \Big(\int\limits_A \gamma (|z-x|) \, dz \Big)
\Big(\int\limits_{B(x_0,r)^c} \gamma (z-x_0) \, dz
\Big)^{-1} 
\end{align*}
for every Borel set $A \subset B(x_0,r)^c$. With some positive
constant $c_5 \geq 1$ depending on $k$ we obtain for every bounded function
$f\colon \R^d\rightarrow \R$  
\begin{align*}
& M(x_0,r) \sup\limits_{x \in B(x_0,r/2)} 
\int\limits_{B(x_0,r)^c} f^-(z) m(x,z-x) \, dz \\
 &\leq c_5 \Big( \int\limits_{B(x_0,r)^c} \gamma (|y-x_0|) \, dy \Big)^{-1}
\sup\limits_{x \in
B(x_0,r/2)} \int\limits_{B(x_0,r)^c} f^-(z) \gamma(|z-x|) \, dz \,.
\end{align*}
This observation together with the
main assumption of the theorem ensures that there exist constants $c_1, c_2
\geq 1$ such that for
every such $x_0\in \R^d$, $r\in (0,1)$ and every bounded function $f\colon
\R^d\rightarrow \R$ which is
non-negative in $B(x_0,r)$  and harmonic in $B(x_0,r)$ the following
estimate holds 
\begin{align}\label{eq:hoelder-harnack-ass_corr}
\sup\limits_{B(x_0,r/4)} f  \leq c_1
\inf\limits_{B(x_0,r/4)} f + c_2 \sup\limits_{x \in
B(x_0,r/2)} \int\limits_{B(x_0,r)^c} f^-(z) \nu^x_r(dz) 
\,.
\end{align} 
We aim to apply Lemma 11 from
\cite{KaDy11}. Note
that it is not important for the
application of \cite[Lemma 11]{KaDy11} whether harmonicity is defined
with respect to an operator $\mathcal{L}$ or some Dirichlet form. Assumption
(\ref{eq:assum_gamma})
implies that there are $c_6 \geq 1$ and $R_0>1$ such that for every $R > R_0$,
$r\in
(0,1)$ and $x \in B(x_0,r/2)$  
\begin{align}\label{eq:est_U}
 \int\limits_{B(x_0,R)^c} \gamma(|z-x|) \, dz \leq c_6 R^{-\sigma_1} 
\end{align}
Moreover, there is $c_7 \geq 1$ with  
\begin{align}\label{eq:est_Ukomp}
\Big(\int\limits_{B(x_0,r)^c} \gamma(|z-x_0|) \, dz\Big)^{-1} \leq c_7 r^{\sigma_2}
\,.
\end{align}
Estimates (\ref{eq:est_U}) and (\ref{eq:est_Ukomp}) imply:
\begin{align*}
& \exists c_8 \geq 1 \; \forall r \in (0,1) \; \exists j_0 \geq 1 \; \forall
j
\geq j_0 \; \forall x \in B(x_0,\tfrac r2): \\
& \qquad 
\nu^x_r\big(B(x_0,2^{j} r)^c\big) \leq c_8
(2^{j}r)^{-\sigma_1} r^{\sigma_2} \leq c_8 2^{-\sigma j} \,.
\end{align*}
Recall that we assumed $\sigma_1 \leq \sigma_2$. Note that
$2^{-\sigma} < 1$ and
$c_8^{1/j} \to
1$ for $j \to \infty$. We finally proved
\begin{align}\label{eq:assum_nu_r}
\sup\limits_{0<r<1} \limsup\limits_{j \to \infty} (\eta_{r,j})^{1/j} < 1, \quad \text{ where }
\eta_{r,j} :=
\sup\limits_{x \in B(x_0,r/2)} \nu^x_r(B(x_0,2^{j}
r)^c ) < \infty \,.
\end{align} Lemma 11 from
\cite{KaDy11} can be applied. The proof is complete.
\end{proof}

\section*{Appendix}

We explain the geometric arguments behind the proof of Proposition \ref{prop:rhi}

Given $\eta \in S^{d-1}$ and $\rho >0$ we define a cone $V(\eta,\rho) \subset \R^d$ as follows. Set
\begin{align*}
 S(\eta, \rho) &= \big( B(\eta, \rho) \cup B(-\eta, \rho) \big) \cap S^{d-1} \text{ and } V(\eta, \rho) = \{ x \in \R^d | x \ne 0, \tfrac{x}{|x|} \in S(\eta, \rho) \} \,.
\end{align*}

From now on, we keep $\eta \in S^{d-1}$ and $\rho >0$ fixed and write $V$ instead of $V(\eta, \rho)$. Choose $\vartheta \in (0, \frac{\pi}{2}]$ so that $\rho^2 = 2(1-\cos\vartheta)$.

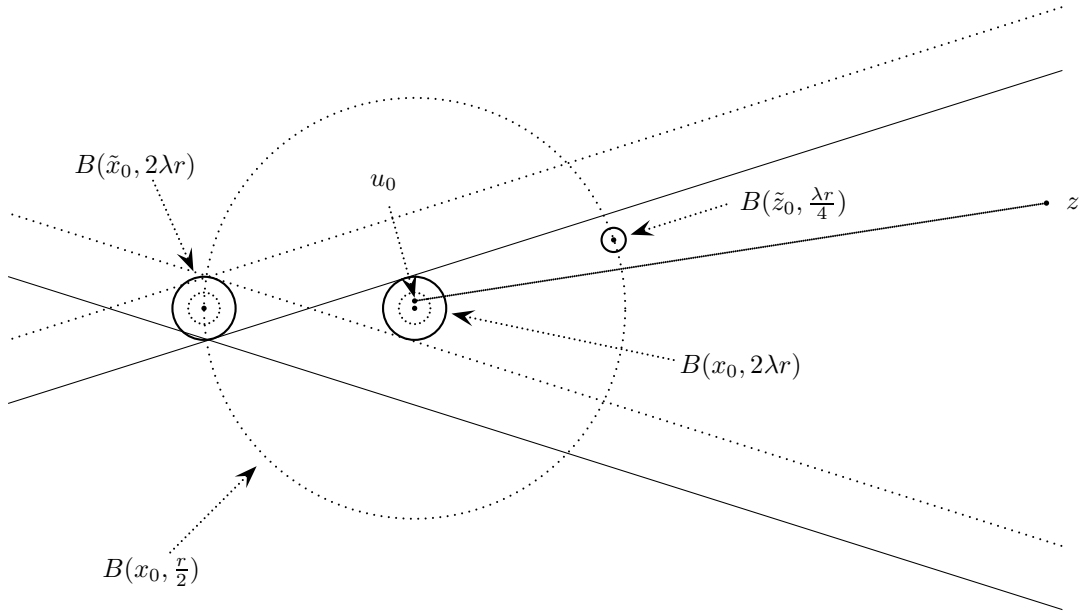
\begin{figure}[ht!]
 \centering 
\psset{unit=0.70cm}
\fontsize{10}{6}\selectfont
\begin{pspicture*}(-4,-6)(20,6)
\psset{linewidth=\pslinewidth}
\psset{dotsize=2pt 0}
\pstGeonode[PointName=none](0, 0){C1}(4,0){C2}
\pstCircleOA[Radius=\pstDistVal{0.6}]{C1}{}
\pstCircleOA[Radius=\pstDistVal{0.6}]{C2}{}
\pstCircleOA[Radius=\pstDistVal{0.3},linestyle=dotted,
dotsep=0.05]{C1}{}
\pstCircleOA[Radius=\pstDistVal{0.3},linestyle=dotted,
dotsep=0.05]{C2}{}
\pstMiddleAB[PointSymbol=none, PointName=none]{C1}{C2}{T}
\pstCircleOA[Radius=\pstDistVal{4},linecolor=black,linestyle=dotted,
dotsep=0.1]{C2}{}
\pstInterCC[RadiusA=\pstDistVal{0.61},
             CodeFigB=false, PosAngleB=45,PointName=none,PointSymbol=none]
{C1}{}{T}{}{A}{B}
\pstInterCC[RadiusA=\pstDistVal{0.61},
             CodeFigB=false, PosAngleB=45,PointName=none,PointSymbol=none]
{C2}{}{T}{}{C}{D}
\pstLineAB[linecolor=black,linestyle=dotted,
dotsep=0.1, nodesepA=17,nodesepB=8]{A}{C}
\pstLineAB[linecolor=black,nodesepA=17,nodesepB=8,,linewidth=0.02]{B}{D}
\pstTranslation[PointName=none,PointSymbol=none]{T}{B}{A}
\pstLineAB[linecolor=black,linewidth=0.02,nodesepA=19,nodesepB=4]{A'}{B}
\pstTranslation[PointName=none,PointSymbol=none]{T}{A}{B}
\pstLineAB[linecolor=black,linestyle=dotted,
dotsep=0.1, nodesepA=19,nodesepB=4]{B'}{A}
\pstGeonode(16,2){z}
\pstGeonode[PointName=none](4,0.14){u_0}
\pstLineAB[linecolor=black,linestyle=dotted,
dotsep=0.00]{u_0}{z}
\pstGeonode[PointName=none](7.78,1.3){zz}
\pstCircleOA[Radius=\pstDistVal{0.23}]{zz}{}
\pstGeonode[PointName=none,PointSymbol=none](8.1,1.4){t1}(10,2){t2}
\ncline[linecolor=black,linestyle=dotted,
dotsep=0.06, arrowscale=2]{->}{t2}{t1}
\rput(11.2,2){$B(\tilde{z}_0,\frac{\lambda r}{4})$}

\pstGeonode[PointName=none,PointSymbol=none](9,-1){t3}(4.7,-0.1){t4}
\ncline[linecolor=black,linestyle=dotted,
dotsep=0.06, arrowscale=2]{->}{t3}{t4}
\rput(10.2,-1.1){$B(x_0,2\lambda r)$}
\pstGeonode[PointName=none,PointSymbol=none](-1,2.5){t5}(-0.2,0.7){t6}
\ncline[linecolor=black,linestyle=dotted,
dotsep=0.06, arrowscale=2]{->}{t5}{t6}
\rput(-1.3,2.7){$B(\tilde{x}_0,2\lambda r)$}
\pstGeonode[PointName=none,PointSymbol=none](-0.7,-4.7){t7}(1,-3){t8}
\ncline[linecolor=black,linestyle=dotted,
dotsep=0.06, arrowscale=2]{->}{t7}{t8}
\rput(-1,-5){$B(x_0,\frac{r}{2})$}
\pstGeonode[PointName=none,PointSymbol=none](3.5,2){t9}(4,0.2){t10}
\ncline[linecolor=black,linestyle=dotted,
dotsep=0.06, arrowscale=2]{->}{t9}{t10}
\rput(3.41,2.4){$u_0$}
\end{pspicture*}
\caption{\label{fig:two} The choice of $\tilde{x}_0$ and $\tilde{z}_0$.}
\end{figure}

Using a simple geometric argument one can establish the following fact: 

Let $\lambda \in (0, \frac{\sin \vartheta}{8})$, $x_0 \in \R^d$, $r \in (0,2)$,
$u_0 \in B_{\lambda r}(x_0)$ and $z \in B(x_0,\frac{3r}{2})^c$. Assume $z \in
u_0 + V$. Set $\widetilde{x_0}=x_0 - \frac r2 \xi \in \partial B(x_0, \frac r2)$
where $\xi \in \{+\eta, - \eta\}$ is chosen so that $\langle z-u_0, \xi \rangle
> 0$, see Figure \ref{fig:two}. Then the choice of $\lambda$ implies

\begin{itemize}
\item[(1)] $ B(\tilde{x}_0, 2\lambda r) \subset \bigcap\limits_{u \in B(x_0, 2\lambda r)} (u + V)$ \,.
\end{itemize}

Moreover, there is $\tilde{z}_0 \in \partial B(x_0, \frac r2)$ such that 
\begin{itemize}
\item[(2)] $B(\tilde{z}_0, \frac{\lambda r}{4}) \subset \bigcap\limits_{v \in B(\tilde{x}_0, 2\lambda r)} (v + V)$ \,, 
\item[(3)] $z \in \bigcap\limits_{w \in B(\tilde{z}_0, \frac{\lambda r}{4})} (w + V)$ \,, 
\item[(4)] $|z-\tilde{z}_0| < |z-x_0|$ \\ and thus $|z-w| < |z-u|$ for all $u \in B(x_0, 4 \lambda r), w \in B(\tilde{z}_0, \frac{\lambda r}{4})$ \,.
\end{itemize}
These conditions assure that the Markov jump process under consideration has a strictly positive probability to jump from a neighborhood of $x_0$ via neighborhoods of $\tilde{x}_0$ and $\tilde{z}_0$ to $z$. One could avoid the introduction of $\tilde{z}_0$ and let the process jump directly from the neighborhood of $\tilde{x}_0$ to $z$ but this would result in a slightly stronger assumption than (J2).

\def\cprime{$'$}

% \bibliographystyle{alpha}
% \bibliography{myrefs}

\end{document}